\documentclass{amsart}

\usepackage{xcolor}
\usepackage{fancyhdr}
\usepackage{amssymb,amsfonts,amsmath,amsthm,tikz,tikz-cd,graphicx,url,mathtools}
\usepackage[all,arc]{xy}
\usepackage{enumerate}
\usepackage{mathrsfs}
\usepackage[colorlinks=true]{hyperref}
\usepackage[alphabetic]{amsrefs}
\usepackage{comment}

\newtheorem{thm}{Theorem}[section]

\newtheorem{prop}[thm]{Proposition}
\newtheorem{lem}[thm]{Lemma}

\newtheorem{quest}[thm]{Question}

\theoremstyle{definition}
\newtheorem{defn}[thm]{Definition}

\theoremstyle{remark}
\newtheorem{rem}[thm]{Remark}

\newcommand{\Z}{\mathbb{Z}}

\newcommand{\CC}{\mathbb{C}}

\newcommand{\rsa}{\rightsquigarrow}

\newcommand{\cag}{\mathcal{G}}

\newcommand{\cab}{\mathcal{B}}
\newcommand{\cac}{\mathcal{C}}

\newcommand{\scrd}{\mathscr{D}}

\DeclareMathOperator{\SO}{SO}

\DeclareMathOperator{\SHI}{SHI}

\DeclareMathOperator{\KHI}{KHI}

\newcommand{\uSHI}{\underline{\SHI}}

\newcommand{\upsi}{\underline{\Psi}}

\usepackage{hyperref}

\title{A Homological Action on Sutured Instanton Homology}

\author{Hongjian Yang}
\date{}
\email{yhj@stanford.edu}
\address{Department of Mathematics, Stanford University, Stanford, CA, 94305}

\begin{document}
	
\maketitle

\begin{abstract}
We define a homological action on sutured instanton Floer homology. This action is well-defined up to scalars, and behaves well under connected sums and sutured manifold decompositions. As an application, we show that instanton knot homology detects link splitting for two-component links. 
\end{abstract}

\section{Introduction}

Sutured manifold theory, introduced by Gabai \cite{gabai1983foliations}, has shown to be a very useful machinery to study knots and $3$-manifolds. Its interaction with Floer theory leads to even more interesting results. It is done in two ways: using Gabai's result \cite{gabai1987foliations} to produce taut foliations and then obtain non-vanishing results by relations to contact structures \cite{eliashberg1998confoliations,eliashberg2004few,kronheimer2004witten,ozsvath2004holomorphic}, or using the sutured manifold hierarchies directly. 

The second approach was first realized by Juh\'asz \cite{juhasz2006holomorphic} in the setting of Heegaard Floer theory, and later by Kronheimer and Mrowka \cite{kronheimer2010knots} for monopole and instanton Floer theories. In particular, it leads to series of topological applications, including fibered knot detection for knot Floer homology \cite{ghiggini2008knot,ni2007knot}, and unknot detection for Khovanov homology \cite{kronheimer2011khovanov}.

The homological action on Floer homology groups can often provide more information about topological objects. For Floer theories of closed $3$-manifolds, such actions have been constructed in \cite{donaldson2002floer, ozsvath2004holomorphic,Kronheimer2007MonopolesAT}. A similar action was also constructed for Khovanov homology \cite{hedden2013khovanov}. Ni \cite{ni2014homological} constructed a homological action on sutured (Heegaard) Floer homology. The main purpose of this paper is to construct a similar action on sutured instanton homology.

\begin{thm}\label{thm action}
	Let $(M,\gamma)$ be a balanced sutured manifold. Then there is an action of $H_1(M,\partial M)$ on the sutured instanton Floer homology $\SHI(M,\gamma)$, well-defined up to multiplication by scalars.
\end{thm}

In contrast to the Heegaard Floer setting, our approach in this paper utilizes the known action on closed 3-manifolds. To verify the well-definedness, we need to compare the action on different closures, which relies on Baldwin--Sivek's naturality result \cite{baldwin2015naturality}. Our action is well-defined up to scalars, as the naturality of $\SHI$ is known to hold only up to scalars.

\begin{thm}\label{thm well behaved}
The action defined in Theorem \ref{thm action} is well-behaved under sutured manifold decompositions and connected sums.
\end{thm}

We will clarify the exact meaning of this theorem in a sequence of propositions in Section \ref{sec main result}. 

The homological action allows us to use instanton knot homology to detect split links. Recall that for a link $L\subset Y$, we can form a balanced sutured manifold $Y(L)=(M,\gamma)$ by taking $M$ to be the link complement and the sutures to be two oppositely-oriented meridians for each component of $L$. The \textit{instanton knot homology} $\KHI(Y,L)$ is then defined as $\SHI(M,\gamma)$. For a link $L$ in $S^3$, it is simply denoted by $\KHI(L)$. For a two-component link $L$ in $S^3$, let $X$ denote the action of a generator of $H_1(M,\partial M)=\Z$. We then have:

\begin{thm}\label{thm splitting}
	Let $L$ be a two-component link in $S^3$. Then $L$ is split if and only if $\KHI(L)$ is free as a $\CC[X]/X^2$-module.
\end{thm}

Detection results for link splitting have been obtained for Khovanov homology, Heegaard Floer homology \cite{lipshitz2022khovanov}, and link Floer homology {\cite{wang2021link}}. A similar idea is also used in \cite{li2022floer} to detect the unknot in sutured manifolds. See also \cite{ni2013nonseparating,hom2022dehn} for results on (non-)separating sphere detection.  Our proof of Theorem \ref{thm splitting} is a formal adaptation of the proof in \cite[Theorem 1.1]{wang2021link}. 

\begin{rem}
	It is also expected to obtain a parallel statement to \cite[Theorem 1.4]{wang2021link} with a stronger result that allows us to detect whether the algebraic intersection number, rather than just its parity, of a $1$-cycle and some embedded $2$-sphere, is zero. However, a major obstacle is the lack of a known non-vanishing theorem (and furthermore, a non-free theorem analogous to \cite[Lemma 3.2]{wang2021link}) for framed instanton homology of closed $3$-manifolds.
\end{rem}

\begin{quest}
	Let $Y$ be a closed $3$-manifold. Is the framed instanton homology $\operatorname{I}^\sharp(Y)$ always non-zero? Further, if $Y$ is irreducible, is $\operatorname{I}^\sharp(Y)$ always not a free $\CC[X]/X^2$-module with respect to the action of any $\zeta\in H_1(Y)$?
\end{quest}

\begin{rem}
	While stated in the context of instantons, our results also hold for monopoles with the $\CC$-coefficient replaced by a Novikov ring. The proofs are similar except that we replace the eigenvalue decomposition by the spin$^c$ decomposition, and take a little bit more care when forming the closure. In this case, the action is expected to equivalent to the one on sutured Floer homology through the isomorphism described in \cite{baldwin2021equivalence}.
\end{rem}

\subsection*{Acknowledgements}The author is indebted to Zhenkun Li for introducing this problem to him and sharing his expertise in this area, and his advisor Ciprian Manolescu for guidance in writing the paper and constant encouragement. He would also like to thank Xingpei Liu and Fan Ye for many helpful discussions.

\section{Preliminaries}

In this section, we review the construction of sutured instanton Floer homology. In particular, we address the naturality issue of sutured instanton Floer homology, treated by Baldwin and Sivek \cite{baldwin2015naturality}.

We first review instanton Floer homology in the closed case very briefly. Let $Y$ be closed $3$-manifold and $w\to Y$ be an Hermitian line bundle such that $c_1(w)$ has odd pairing with some integer homology class. Let $E\to Y$ be a $\operatorname{U}(2)$-bundle with an isomorphism $\theta\colon \Lambda^2E\to w$. Let $\cac$ be the space of $\SO(3)$-connections on $\operatorname{ad}(E)$, and $\cag$ be the group of determinant-$1$ gauge transformations of $E$ (i.e. automorphisms of $E$ that respect $\theta$). Roughly speaking, the \textit{instanton Floer homology} of $Y$, denoted by $I_*(Y)_{w}$, is a $\Z/8\Z$-graded $\CC$-module arising from the Morse homology of the Chern--Simons functional on $\cab=\cac/\cag$ \cite{donaldson2002floer}. A homology class $\alpha\in H_k(Y)$ gives a cohomology class $\mu(\alpha)\in H^{4-k}(\cab)$ by slant products, and hence an operator \[\mu(\alpha)\colon I_*(Y)_w\to I_{*-4+k}(Y)_w.\]For homology classes $\alpha,\beta$, we have \[\mu(\alpha)\mu(\beta)=(-1)^{|\alpha||\beta|}\mu(\beta)\mu(\alpha).\]

 Kronheimer and Mrowka \cite{kronheimer2010knots} define the \textit{sutured instanton Floer homology} $\SHI(M,\gamma)$ for a balanced sutured manifold $(M,\gamma)$. \begin{defn}
 	A \textit{balanced sutured manifold} $(M,\gamma)$ consists of a compact, oriented, smooth $3$-manifold $M$ together with a collection of oriented circles $\gamma\subset \partial M$, called the \textit{sutures}. Let $A(\gamma)$ be a closed tubular neighbourhood of $\gamma$, and $R(\gamma)$ be the closure of $\partial M\backslash A(\gamma)$. We require that \begin{enumerate}
 		\item $M$ and $R(\gamma)$ have no closed components;
 		
 		\item $\partial A(\gamma)$ is oriented in the same way as $\gamma$, and $R(\gamma)$ is oriented such that $\partial R(\gamma)=\partial A(\gamma)$ as oriented manifolds;
 		
 		\item if we define $R_+(\gamma)$ (resp. $R_-(\gamma)$) as the subset of $R(\gamma)$ where the orientation is same (resp. the opposite) as the boundary orientation on $\partial M$, then $\chi(R_+)=\chi(R_-)$.
 	\end{enumerate}
 \end{defn}
 
  For backgrounds on (balanced) sutured manifolds and surface decompositions, we refer the reader to \cite{gabai1983foliations,juhasz2006holomorphic,juhasz2008floer}. 
  
  Here are two examples of balanced sutured manifolds that we will be interested in. For a link $L$ in a $3$-manifold $Y$, we can from $Y(L)$ as we do before Theorem \ref{thm splitting}. For a closed $3$-manifold $Y$, we can form a balanced sutured manifold $Y(n)$ by removing $n$ $3$-balls and adding one suture on each boundary component. In particular, the sutured instanton homology of $Y(1)$ gives the \textit{framed instanton homology} $I^\sharp(Y)$ \cite{kronheimer2011knot}.
  
  Now let $A(\gamma)$ be a closed tubular neighbourhood of the suture $\gamma$ in $\partial M$. Choose a compact, connected, oriented surface $T$ with one marked point $t_0$, called the \textit{auxiliary surface}, such that $g(T)>0$ and that $|\partial F|=|\gamma|$. Let \[h\colon\partial T\times[-1,1]\to A(\gamma)\]be an orientation-preserving homeomorphism, and form the manifold \[M'=M\cup_{h}(T\times[-1,1]),\]called the \textit{preclosure}, by gluing along the boundary and rounding corners. The balanced condition implies $M'$ has two boundary components with the same genus. We then close up $M'$ by identifying them using a diffeomorphism $\phi$ that maps $t_0$ to itself. The resulting closed $3$-manifold $Y$ contains a distinguished surface $R$ and a standard circle $\alpha$ through $t_0$. Let $w$ be the line bundle with first Chern class  Poincar\'e dual to $\alpha$. We then define $\SHI(M,\gamma)$ as $I_*(Y|R)_w$, the $(2g(R)-2,2)$-generalized eigenspace of $(\mu(R),\mu(y))$ in $I_*(Y)_w$. When $g(R)>1$, it suffices to take the $(2g(R)-2)$-eigenspace of $\mu(R)$. 
 
 Using this definition, it is shown that the \textit{isomorphism class} of the $\CC$-module is an invariant of $(M,\gamma)$. However, this viewpoint alone does not provide an actual $\CC$-module, which is required for defining an action on it. The challenge lies in comparing the instanton homology for different closures. To address this issue, Baldwin and Sivek \cite{baldwin2015naturality} introduced a refined version of closure, which allows for meaningful comparisons.

\begin{defn}
	An \textit{odd closure} of a balanced sutured manifold $(M,\gamma)$ is a tuple $(Y,R,r,m,\alpha)$, where \begin{enumerate}
		\item $Y$ is a closed, oriented, smooth $3$-manifold;
		
		\item $R$ is a closed, oriented, smooth surface of genus at least $2$;
		
		\item $r\colon R\times[-1,1]\hookrightarrow Y$ is a smooth, orientation-preserving embedding;
		
		\item $m\colon M \hookrightarrow Y\backslash \operatorname{int}(\operatorname{im}(r))$ is a smooth, orientation-preserving embedding such that \begin{enumerate}
			\item it extends to a diffeomorphism \[M\cup_h (F\times [-1,1])\to Y\backslash \operatorname{int}(\operatorname{im}(r))\]for auxiliary surface $F$ and gluing homeomorphism $h$ in the sense of Kronheimer--Mrowka;
			
			\item it restricts to an orientation-preserving embedding \[R_+(\gamma)\backslash A(\gamma)\hookrightarrow r(R\times\{-1\});\]
		\end{enumerate}
	
	\item $\alpha$ is an oriented, smoothly embedded curve in $Y$ such that \begin{enumerate}
		\item it is disjoint from $\operatorname{im}(m)$;
		
		\item it intersects $r(R\times[-1,1])$ in a product arc $r(\{p\}\times [-1,1])$ for some $p\in R$.
	\end{enumerate}
	\end{enumerate}
\end{defn}

\begin{defn}
	A \textit{marked odd closure} of $(M,\gamma)$ is a tuple $(Y,R,r,m,\eta,\alpha)$ such that \begin{enumerate}
		\item $(Y,R,r,m,\alpha)$ is an odd closure of $(M,\gamma)$;
		
		\item $\eta$ is an oriented, homologically essential, smoothly embedded curve in $R$.
	\end{enumerate}
\end{defn}

\begin{defn}
	Let $\scrd=(Y,R,r,m,\eta,\alpha)$ be a marked odd closure of $(M,\gamma)$. Let $w$ and $u$ be the line bundles with first Chern classes Poincar\'e dual to $\alpha$ and $\eta$, respectively. The \textit{(untwisted) sutured instanton Floer homology} of $\scrd$ is a $\CC$-module defined as \[\SHI(\scrd)= I_*(Y|r(R\times\{0\}))_{\alpha}\coloneqq I_*(Y|r(R\times\{0\}))_w.\]The \textit{twisted sutured instanton Floer homology} is a $\CC$-module defined as \[\uSHI(\scrd)= I_*(Y|r(R\times\{0\}))_{\alpha+\eta}\coloneqq I_*(Y|r(R\times\{0\}))_{u\otimes w}.\]
\end{defn}

We will write $I_*(Y|R)$ to indicate $I_*(Y|r(R\times\{0\}))$ for short. 

\begin{thm}[{\cite[Theorem 9.16]{baldwin2015naturality}}]\label{thm BS naturality}
	Let $\scrd,\,\scrd'$ be marked odd closures of $(M,\gamma)$. Then there is a canonical isomorphism \[\underline{\Psi}_{\scrd,\scrd'}\colon\uSHI(\scrd)\to\uSHI(\scrd'),\]well-defined up to units in $\CC$. Further, let $\scrd_0,\scrd_1,\scrd_2$ be marked odd closures of $(M,\gamma)$. Then  \[\underline{\Psi}_{\scrd_0,\scrd_2}=\underline{\Psi}_{\scrd_1,\scrd_2}\circ\underline{\Psi}_{\scrd_0,\scrd_1}\]up to units in $\CC$.
\end{thm}

For our purpose, we sketch the proof of Theorem \ref{thm BS naturality}. The map $\underline{\Psi}_{\scrd,\scrd'}$ is defined in two steps: first for marked odd closures of the same genus, and then for marked odd closures whose genera differ by one. The first step is done by reducing to the cases of positive and negative Dehn twists, and the second step is done by applying a torus excision theorem to decrease genus. Finally, we define the map $\underline{\Psi}_{\scrd,\scrd'}$ in general by composing the maps obtained in the two steps described above.

In the language of \cite{baldwin2015naturality}, $\uSHI$ defines a functor from the category of balanced sutured manifolds (with morphisms given by isotopy classes of diffeomorphisms) to the category of $\CC$-projective transitive systems, and the module class recovers the original $\SHI$.

\section{Construction of the action}\label{cons}

Let $(M,\gamma)$ be a balanced sutured manifold. In this section, we define the action of $H_1(M,\partial M)$ on $\uSHI(M,\gamma)$, and show that it is well-defined up to scalars.

Let $\zeta\in H_1(M,\partial M)$ be a $1$-cycle. We start by defining the action of $\zeta$ for each marked odd closure  $\scrd=(Y,R,r,m,\eta,\alpha)$ of $(M,\gamma)$. To do this, we first represent $\zeta$ by a balanced tangle $T$ in $M$ in the sense of \cite{xie2019instanton}. That is, we require that \[|T\cap R_+(\gamma)|=|T\cap R_-(\gamma)|\]and that \[\pi\circ r^{-1}\circ m(T\cap R_+(\gamma))=\pi\circ r^{-1}\circ m(T\cap R_-(\gamma)),\]where $\pi\colon R\times I\to R$ is the projection. Such $T$ exists since being balanced means each boundary component of $M$ must contain a suture, and we can move the endpoints of $T$ freely on each boundary component. We require in further that $T$ is disjoint with $\alpha$ and $\eta$.

We then define an action $X_\zeta$ of $\zeta$ on $\uSHI(\scrd)=I_*(Y|R)_{u\otimes w}$ as follows. Close up $T$ in $Y$ by gluing $T$ with a product tangle \[r((\pi\circ r^{-1}\circ m(T\cap R_+(\gamma)))\times I).\]Denote the resulting $1$-cycle by $\widetilde{T}$, and then $\mu([\widetilde{T}])$ acts on $I_*(Y)_{u\otimes w}$.  As $\mu(R)$ and $\mu([\widetilde{T}])$ commute, we can define $X_\zeta$ as the restriction of this action on $I_*(Y|R)_{u\otimes w}$. It is clearly well-defined on $\uSHI(\scrd)$ since homologous $T$ in $(M,\partial M)$ yields homologous $\widetilde{T}$ in $Y$.

We need to show this action is well-defined. Recall that the map $\upsi_{\scrd,\scrd'}$ is defined in two steps: for the same genus, and for genera differing by one. So it suffices to check the maps respect the actions in these two cases.

\begin{prop}\label{prop same genus}
	Let $(M,\gamma)$ be a balanced sutured manifold, and let $\zeta\in H_1(M,\partial M)$. Let $\scrd$ and $\scrd'$ be two marked odd closure of $(M,\gamma)$ with the same genus. Then the canonical isomorphism \[\Phi_{\scrd,\scrd'}\colon \uSHI(\scrd)\to\uSHI(\scrd'),\]defined in \cite[Section 9.3.1]{baldwin2015naturality}, respects the action of $\zeta$ up to scalars. 
\end{prop}

\begin{proof}
	For the sake of exposition, we write \[\scrd=\scrd_1=(Y_1,R_1,r_1,m_1,\eta_1,\alpha_1),\]\[ \scrd'=\scrd_2=(Y_2,R_2,r_2,m_2,\eta_2,\alpha_2).\]Roughly speaking, the map $\upsi_{\scrd,\scrd'}$ is defined as a composition \[\upsi_{\scrd,\scrd'}\coloneqq \Theta^C_{(Y_1)_+Y_2}\circ I_*(X_+|R_1)_{(\alpha_1+\eta_1)\times[0,1]}\circ\left(I_*(X_-|R_1)_{(\alpha_1+\eta_1)\times[0,1]}\right)^{-1}.\]Here \begin{itemize}
		\item 	$C$ is a fixed diffeomorphism \[C\colon Y_1\backslash\operatorname{int}(\operatorname{im}(r_1))\to Y_2\backslash\operatorname{int}(\operatorname{im}(r_2)),\]which restricts to $m_2\circ m_1^{-1}$ on $m_1(M\backslash N(\gamma))$ for some tubular neighbourhood $N(\gamma)$ and sends $\alpha_1\cap(Y_1\backslash\operatorname{int}(\operatorname{im}(r_1)))$ to $\alpha_2\cap(Y_2\backslash\operatorname{int}(\operatorname{im}(r_2)))$;
		
		\item $X_-\colon (Y_1)_-\to Y_1$ and $X_+\colon (Y_1)_-\to (Y_1)_+$ are certain cobordisms, corresponding to some $\pm 1$-surgeries on $Y_1$. The manifold $(Y_1)+$ is diffeomorphic to $Y_2$ via a map that restricts to $C$ on $Y_1\backslash\operatorname{int}(\operatorname{im}(r_1))$ and to a specific diffeomorphism on a neighbourhood of $r_1(R_1\times{0})$. This map belongs to a unique isotopy class according to \cite[Appendix A]{baldwin2015naturality}, and $\Theta^C_{(Y_1)_+Y_2}$ is the isomorphism on instanton homology associated with this isotopy class.
	\end{itemize}

	Hence, it suffices to show that the maps induced by the cobordisms $X_-$ and $X_+$ and the map $\Theta^C_{(Y_1)_+Y_2}$ respect the action of $\zeta$.  
	
	We can assume that $X_-$ and $X_+$ are corresponding to a single Dehn twist. In this case, the cobordism $X_-$ is constructed by taking the product cobordism $Y\times I$, where $Y=(Y_1)_-$, and attaching a $(-1)$-framed $2$-handle along a curve $r_1(a\times{t})\times{1}$ in $Y\times {1}$, where $a$ is a simple closed curve in $R_1$ and $t\in (-1,1)$. 
	
	Since the Dehn twist is a local operation that does not affect the embedding $m_1$, we can consider the tangle closure $\widetilde{T}$, which appears in the definition of the action, as living in both $Y_1$ and $(Y_1)-$. These two instances of $\widetilde{T}$ are related by the product $\widetilde{T}\times I$ contained in $X-$. The action $\mu([\widetilde{T}])$ on $I_*(Y_1)$ (resp. $I_*((Y_1)_-)$) coincides with the cobordism map $I_*(Y_1\times I,[\widetilde{T}]\times I)$ (resp. $I_*((Y_1)_-\times I,[\widetilde{T}]\times I)$). By the functoriality of instanton Floer homology (cf. \cite[Theorem 3.4.4]{Kronheimer2007MonopolesAT} in the context of monopoles), \[I_*(Y_1\times I,[\widetilde{T}]\times I)\circ I_*(X_-,[\widetilde{T}]\times I)= I_*(X_-,[\widetilde{T}]\times I)\circ I_*((Y_1)_-\times I,[\widetilde{T}]\times I).\]The action descends onto $I_*(Y_1|R_1)$ and $I_*((Y_1)_-|R_1)$ in an equivariant manner, as the operator $\mu([\widetilde{T}])$ commutes with $\mu(R)$ and $\mu(y)$. The same argument holds for $I_*(X_+)$.
	
	We now show the map $\Theta^C_{(Y_1)_+Y_2}$ also respects the action. Let \[f\colon (Y_1)+\to Y_2\]be an orientation-preserving diffeomorphism such that $f$ restricts to $m_2\circ m_1^{-1}$ on $m_1(M\backslash N(\gamma))$ to some certain map on $r_1(R\times \{0\})$. We can choose a tangle $T\subset M$ such that it is balanced in both $(Y_1)_+$ and $Y_2$ and form the closures $\widetilde{T}_1$ and $\widetilde{T}_2$ respectively. Recall that $\widetilde{T}_1$ is the concatenation of $m_1(T)$ and a product tangle $T'$ in the image of $r_1$. The diffeomorphism $f$ restricts to an orientation-preserving diffeomorphism $f_1$ from $R_1\times I$ to $R_2\times I$ (more explicitly, their images under $r_1$ and $r_2$), which sends boundaries to boundaries. It induces an isomorphism \[(f_1)_*\colon H_1(R_1\times I,\partial (R_1\times I))\to H_1(R_2\times I,\partial(R_2\times I)),\]which must be the identity as $f_1$ is orientation-preserving and \[H_1(R_1\times I,\partial (R_1\times I))\cong\Z.\] Hence, $f(\widetilde{T}_1)$ is homologous to $\widetilde{T}_2$, and they induce the same action on $I_*(Y_2|R_2)$.
\end{proof}

\begin{prop}\label{prop genera}
	Let $(M,\gamma)$ be a balanced sutured manifold, and let $\zeta\in H_1(M,\partial M)$. Let $\scrd$ and $\scrd'$ be two marked odd closures of $(M,\gamma)$ such that $g(\scrd')=g(\scrd)\pm 1$. Then the canonical isomorphism \[\Phi_{\scrd,\scrd'}\colon \uSHI(\scrd)\to\uSHI(\scrd'),\]defined in \cite[Section 9.3.2]{baldwin2015naturality}, respects the action of $\zeta$  up to scalars. 
\end{prop}

\begin{proof}
	In the case where $g(\scrd')=g(\scrd)-1$, the map is defined as the inverse of $\Phi_{\scrd',\scrd}$, so we only need to prove the result for the case where $g(\scrd')=g(\scrd)+1$. In this situation, the map $\Phi_{\scrd,\scrd'}$ is defined as a composition \[\Phi_{\scrd_3,\scrd_4}\circ\Phi_{\scrd_2,\scrd_3}\circ\Phi_{\scrd_1,\scrd_2},\]where $\scrd_1=\scrd$, $\scrd_4=\scrd'$, $g(\scrd_3)-1=g(\scrd_2)=g(\scrd_1)$. For the sake of exposition, we write \[ \scrd_2=(Y_2,R_2,r_2,m_2,\eta_2,\alpha_2),\] \[\scrd_3=(Y_3,R_3,r_3,m_3,\eta_3,\alpha_3).\]The closure $\mathscr{d}3$ is a ``cut-ready'' closure, as defined in \cite[Section 5.2]{baldwin2015naturality}, modified from $\scrd_4$. Cutting along two tori in $Y_3$ and re-gluing yields a two-component $3$-manifold, one of which is diffeomorphic to a mapping torus of the surface $\Sigma_{1,2}$ (a genus one surface with two boundary components), and the other is $Y_2$. The maps$\Phi_{\scrd_3,\scrd_4}$ and $\Phi_{\scrd_1,\scrd_2}$ are well-defined since they are maps between closures of the same genus, and they are equivariant by Proposition \ref{prop same genus}. The map $\Phi_{\scrd_2,\scrd_3}$ is induced by a merge-type splicing cobordism corresponding to a torus excision, as described in \cite[Theorem 7.7]{kronheimer2010knots}. Since torus excision is a local operation outside $m_3(M)$, a similar functoriality argument as in Proposition \ref{prop same genus} applies here.
\end{proof}

We can now conclude:

\begin{thm}
	Let $(M,\gamma)$ be a balanced sutured manifold, and let $\zeta\in H_1(M,\partial M)$. Let $\scrd$ and $\scrd'$ be two marked odd closures of $(M,\gamma)$.  Then the canonical isomorphism \[\Phi_{\scrd,\scrd'}\colon \uSHI(\scrd)\to\uSHI(\scrd'),\]defined in \cite{baldwin2015naturality}, respects the action of $\zeta$  up to scalars. 
\end{thm}

\begin{proof}
	Let $\scrd=\scrd_0,\scrd_1,\dots,\scrd_n=\scrd'$ be a sequence of marked odd closures of $(M,\gamma)$ such that $|g(\scrd_i)-g(\scrd_{i+1})|\le 1$. Then the map $\Phi_{\scrd_i,\scrd_{i+1}}$ has been defined previously, and $\Phi_{\scrd,\scrd'}$ is defined as the composition of $\Phi_{\scrd_i,\scrd_{i+1}}$. By Proposition \ref{prop same genus} and \ref{prop genera}, it is equivariant.
\end{proof}

At this point, for $\zeta\in H_1(M,\partial M)$, we have obtained an action $X_\zeta$ on $\SHI(M,\gamma)$. It satisfies $X_\zeta^2=0$, and hence gives a $\CC[X]/X^2$-module structure on $\SHI(M,\gamma)$. We can make sense of the freeness of this module even though the action is only well-defined up to scalars.

\begin{rem}\label{rem original closure}
	After establishing the well-definedness of the action, we can also describe it using the original notion introduced by Kronheimer and Mrowka, as follows. Let $(M,\gamma)$ be a balanced sutured manifold, and let $\zeta\in H_1(M,\partial M)$. Represent $\zeta$ by a tangle $T$ such that \[|T\cap R_+(\gamma)|=|T\cap R_-(\gamma)|.\] Choose an auxiliary surface $T$ and construct the preclosure $M'=M\cup (T\times I)$ as usual, which has two boundary components $\overline{R}_+$ and $\overline{R}_-$ of the same genus. Next, select a diffeomorphism $h\colon \overline{R}_+\to \overline{R}_-$ such that $h(T\cap R_+(\gamma))=T\cap R_-(\gamma)$. By employing $h$ to close up $M'$, we obtain a closure $(Y,R)$ along with a $1$-cycle $\widetilde{T}\subset Y$. The action $\mu(\widetilde{T})$ provides a $\CC[X]/X^2$-module structure on $I_*(Y|R)$, which coincides with the previously defined action.
\end{rem}

\section{Main results}\label{sec main result}

In this section, we study the behaviour of the action under connected sums and sutured manifold decompositions. We state Theorem \ref{thm well behaved} more clearly in Proposition \ref{prop connected sum} and Proposition \ref{prop sutured decomposition}.

We start with some simple observations. The first one originates from \cite{gabai1987foliations,juhasz2006holomorphic} as product disc decompositions, and \cite{baldwin2016contact} expresses in the language of contact handle attachments.

\begin{lem}\label{lem 1 handle}
	Let $(M,\gamma)$ be a balanced sutured manifold, and let $\zeta\in H_1(M,\partial M)$. Let $(M',\gamma')$ be the result of $(M,\gamma)$ with one contact $1$-handle attached. Let \[\iota_*\colon H_1(M,\partial M)\to H_1(M',\partial M)\]be the natural induced map. Then we have an isomorphism of $\CC[X]/X^2$-modules\[\SHI(M',\gamma')\cong\SHI(M,\gamma).\]Here the actions are given by $X_\zeta$ on the right and $X_{\iota_*(\zeta)}$ on the left.
\end{lem}

\begin{proof}
	As explained in \cite{kronheimer2010instanton}, one may use disconnected surfaces that satisfy certain conditions as the auxiliary surface. Hence we can form a marked odd closure $\scrd$ of $(M',\gamma')$ that also serves as a marked odd closure of $(M,\gamma)$. Hence, the $\mu$-action of the tangle closure $[\widetilde{T}]$ gives the action on $\SHI(M,\gamma)$ and $\SHI(M',\gamma')$ simultaneously.
\end{proof}

\begin{lem}\label{lem disjoint union}
	Let $(M_1,\gamma_1)$ and $(M_2,\gamma_2)$ be two balanced sutured manifolds, and let $\zeta_i\in H_1(M_i,\partial M_i)$ ($i=1,2$). Then the action of\[ \zeta=\zeta_1+\zeta_2\in H_1(M_1\sqcup M_2,\partial(M_1\sqcup M_2))\] on \[\SHI(M_1\sqcup M_2,\gamma_1\cup \gamma_2)\cong \SHI(M_1,\gamma_1)\otimes_\CC\SHI(M_2,\gamma_2)\] is given by \[X_{\zeta}(a\otimes b)=X_{\zeta_1}a\otimes b+a\otimes X_{\zeta_2}b.\]
\end{lem}

\begin{proof}
	Again, we can use disconnected auxiliary surfaces to form closures. The result then follows from the statement in the closed case.
\end{proof}

We can now show that the connected sum formula for sutured instanton Floer homology \cite{li2020contact}, holds in an equivariant setting. 

\begin{prop}\label{prop connected sum}
	Let $(M,\gamma)$ be the connected sum of two balanced sutured manifolds $(M_1,\gamma_1)$ and $(M_2,\gamma_2)$, and let $\zeta\in H_1(M,\partial M)$. Then there is an isomorphism of $\CC[X]/X^2$-modules\[\SHI(M,\gamma)\cong\SHI(M_1,\gamma_1)\otimes_\CC \SHI(M_2,\gamma_2)\otimes_\CC\SHI(S^3(2)).\]Here actions on the right are given by some corresponding relative $1$-classes. Further, let $S$ be the $2$-sphere in $M$ along which the connected sum is formed. Then $\SHI(S^3(2))$ is a free $\CC[X]/X^2$-module of rank $1$ if and only if the algebraic intersection number of $\zeta$ and $S$ is non-zero.
\end{prop}

\begin{proof}
	
	As in \cite[Lemma 4.9]{li2020contact}, we have \[(M,\gamma)\cong(M_1\sqcup M_2\sqcup S^3(2),\gamma_1\cup\gamma_2\cup \delta^2)\cup h_1\cup h_2.\]Here $h_i$ is a contact $1$-handle connecting one boundary component of $S^3(2)$ to the boundary of $M_i$. The first statement then follows from Lemma \ref{lem 1 handle} and \ref{lem disjoint union}. Precisely speaking, let $\zeta_1+\zeta_2+\zeta'$ be the image of $\zeta$ in \[H_1(M_1,\partial M_2)\oplus H_1(M_2,\partial M_2)\oplus H_1(S^3(2),\partial (S^3(2))).\]Then the action on the right hand side is given by \[X(a\otimes b\otimes c)=(X_{\zeta_1}a)\otimes b\otimes c+a\otimes(X_{\zeta_2}b)\otimes c+a\otimes b\otimes (X_{\zeta'}c).\]
	
	The remaining task is to calculate the action on $\SHI(S^3(2))$. Notice that $S^1\times S^2(1)$ can be obtained by $S^3(1)$ with one contact $1$-handle attached. We have \[\SHI(S^3(2))\cong \SHI(S^1\times S^2(1))=\operatorname{I}^\sharp(S^1\times S^2).\]
	The last group is calculated in
	\cite{scaduto2015instantons}. As a $\CC$-module, it has two $\CC$-summands lying in grading $2$ and $3$. Two summands are generated by the relative invariants $[S^1\times D^3]^\sharp$ and $[D^2\times S^2]^\sharp$ (see \cite[Section 7.2]{scaduto2015instantons} for the meaning of notations) that are induced by the cobordisms \[S^1\times D^3\colon\emptyset\to S^1\times S^2\]and \[D^2\times S^2\colon\emptyset\to S^1\times S^2\]respectively. Let $\gamma$ be an embedded circle $S^1\times \{pt\}\subset S^1\times S^2$. Then $\gamma$ generates $H_1(S^1\times S^2)$, and $\mu(\gamma)$ maps $[S^1\times D^3]^\sharp$ to $[D^2\times S^2]^\sharp$  by an adaptation of \cite[Theorem 7.16]{donaldson2002floer}. Hence, $\operatorname{I}^\sharp(S^1\times S^2)$ is a free $\CC[X]/X^2$-module of rank $1$ with respect to the action of any nonzero $\eta\in H_1(S^1\times S^2)$. The class $\zeta'$ is nonzero in $H_1(S^3(2),\partial (S^3(2)))=H_1(S^1\times S^2)$ if and only if the algebraic intersection number of $\zeta$ and $S$ is non-zero. The result follows.
\end{proof}

For completeness, we also record the following connected sum formula.

\begin{prop}
	Let $(M,\gamma)$ be the connected sum of a balanced sutured manifold $(M_1,\gamma_1)$ and a closed $3$-manifold $Y$, and let $\zeta\in H_1(M,\partial M)$. Write $\zeta=\zeta_1+\zeta_2$ according to the decomposition $ H_1(M_1,\partial M_1)\oplus H_1(Y)=H_1(M,\partial M)$. Then there is an isomorphism of $\CC[X]/X^2$-modules \[\SHI(M,\gamma)\cong \SHI(M_1,\gamma_1)\otimes_\CC\operatorname{I}^\sharp(Y),\]where $X$ acts on the right by \[X(a\otimes b)=(X_{\zeta_1}a)\otimes b+a\otimes (X_{\zeta_2}b).\]
\end{prop}

\begin{proof}
	We have \[(M,\gamma)\cong(M_1\sqcup Y(1),\gamma_1\cup \delta)\cup h.\]Here $h$ is a contact $1$-handle connecting $M_1$ and $Y$. An argument as in Proposition \ref{prop connected sum} applies here, too.
\end{proof}

We now discuss the interaction of the action and sutured manifold decompositions. The case without the action is treated in \cite{juhasz2008floer,kronheimer2010knots}, and the behaviour of the action is treated in \cite{ni2014homological} for Heegaard Floer theory.

\begin{prop}\label{prop sutured decomposition}
	Let $(M,\gamma)$ be a balanced sutured manifold, $\zeta\in H_1(M,\partial M)$, and let \[(M,\gamma)\overset{S}{\rsa}(M',\gamma')\]be a nice surface decomposition of balanced sutured manifolds. Then $\SHI(M',\gamma')$ is a direct summand of $\SHI(M,\gamma)$ as $\CC[X]/X^2$-modules. Here the action on $\SHI(M',\gamma')$ is defined as $X_{\iota_*(\zeta)}$, where \[\iota_*\colon H_1(M,\partial M)\to H_1(M',\partial M')\cong H_1(M,(\partial M)\cup S)\]is the map induced by inclusion. 
\end{prop}

\begin{proof}
	Recall that the proof of \cite[Proposition 6.9]{kronheimer2010knots} relies on a reduction to a special case as demonstrated in \cite[Lemma 6.10]{kronheimer2010knots}. This reduction only involves the addition of $1$-handles, which preserves the $\CC[X]/X^2$-module structure by Lemma \ref{lem 1 handle}.  
	
	In this special case, we can assume that $S$ has no closed components, and the oriented boundary of $\partial S$ consists of $n$ simple closed curves $C_1^+,C_2^+,\dots,C_n^+$ that are linearly independent in $H_1(R_+(\gamma))$, and $C_1^-,C_2^-,\dots,C_n^-$ that are linearly independent in $H_1(R_-(\gamma))$. 
	
	Choosing the representative tangle $T$ of $\zeta$ appropriately, we can form a closure $Y=Y(M,\gamma)$ using an auxiliary surface $G$ and a gluing diffeomorphism $h$ as usual with additional requirements that $h$ maps $C_i^+$ to $C_i^-$ and $R_+(\gamma)\cap T$ to $R_-(\gamma)\cap T$.  Then $Y$ contains two distinguished closed surfaces: first the usual surface $R$, and second a surface $\overline{S}$ obtained by closing up $S$. The action of $\zeta$ on $\SHI(M,\gamma)$ can be identified with the action of the tangle closure $\mu(\widetilde{T})$ on $I_*(Y|R)$. 
	
	  \begin{figure}[!h]
		\centering
		\begin{tikzpicture}[scale=0.85]
			\begin{scope}[shift={(-2.5,0)}]
				\draw[thick] (0,4)--(-4.3,4) to [out=75,in=180] (-4,5)--(0.3,5);

				\draw[thick] (0,3)--(-4.3,3) to [out=-75,in=180] (-4,2)--(0.3,2);
				
				\draw[thick] (0,4) to [out=75,in=180] (0.3,5) to [out=0,in=90] (0.6,3.5) to [out=-90,in=0] (0.3,2) to [out=180,in=-75] (0,3);
				
				\draw[thick] (-3.7,4) to [out=-105,in=105] (-3.7,3);
				
				\draw (-2.3,4) to [out=75,in=180] (-2,5);
				\draw[dotted] (-2,5) to[out=0,in=75] (-1.7,4);
				\draw (-1.7,4) to [out=-105,in=105] (-1.7,3);
				\draw[dotted] (-1.7,3) to [out=-75,in=0] (-2,2);
				\draw (-2.3,3) to [out=-75,in=180] (-2,2);
				
				\draw[thick,red] (-2.8,4) to [out=85,in=-110] (-1.6,5);
				\draw[thick,red,dotted] (-1.6,5) to [out=-70,in=90] (-1.4,4);
				\draw[thick,red] (-1.4,4) to [out=-90,in=70] (-1.4,3);
				\draw[thick,red,dotted] (-1.4,3) to [out=-110,in=65] (-1.75,2);
				\draw[thick,red] (-1.75,2) to [out=115,in=-85] (-2.8,3);
				\node[right,red] at (-1.4,3.5) {$T$};
				
				\node[left] at (-1.7,3.7) {$S$};

				\draw (-3.0,2.9)--(-4.6,2.3);
				\node at (-4.8,2.2) {$R_+$};
				
				\draw (-1.0,4.1)--(-0.5,4.4);
				\node[above] at (-0.5,4.4) {$R_-$};
				
				\draw[fill=blue,blue] (-2.3,3) circle[radius=0.06];
				\node[above] at (-2.5,2.9) {$C_i^+$};
			\end{scope}
			
			\node at (-1,3.5) {$\rsa$};
			
			\begin{scope}[shift={(4,0)}]
				
				\draw[thick] (-2,4)--(-4.3,4) to [out=75,in=180] (-4,5)--(-1.7,5);
				\draw[thick] (-2,3)--(-4.3,3) to [out=-75,in=180] (-4,2)--(-1.7,2);
				\draw[thick] (-2,4) to [out=75,in=180] (-1.7,5) to [out=0,in=90] (-1.4,3.5) to [out=-90,in=0] (-1.7,2) to [out=180,in=-75] (-2,3);
				\draw[thick] (-3.7,4) to [out=-105,in=105] (-3.7,3);
				
				\draw[thick,shift={(3.2,0)}] (-2,4)--(-4.3,4) to [out=75,in=180] (-4,5)--(-1.7,5);
				\draw[thick,shift={(3.2,0)}] (-2,3)--(-4.3,3) to [out=-75,in=180] (-4,2)--(-1.7,2);
				\draw[thick,shift={(3.2,0)}] (-2,4) to [out=75,in=180] (-1.7,5) to [out=0,in=90] (-1.4,3.5) to [out=-90,in=0] (-1.7,2) to [out=180,in=-75] (-2,3);
				\draw[thick,shift={(3.2,0)}] (-3.7,4) to [out=-105,in=105] (-3.7,3);
				
				\draw[thick,red] (-2.8,4) to [out=85,in=-150] (-1.96,4.6);
				\draw[thick,red,shift={(1.2,0)}] (-2.26,4.6) to [out=30,in=-105] (-1.6,5);
				\draw[thick,red,dotted,shift={(1.2,0)}] (-1.6,5) to [out=-70,in=90] (-1.4,4);
				\draw[thick,red,shift={(1.2,0)}] (-1.4,4) to [out=-90,in=70] (-1.4,3);
				\draw[thick,red,dotted,shift={(1.2,0)}] (-1.4,3) to [out=-110,in=65] (-1.75,2);
				\draw[thick,red,shift={(1.2,0)}] (-2.26,2.4) to [out=-30,in=105] (-1.75,2);
				\draw[thick,red] (-1.96,2.4) to [out=150,in=-85] (-2.8,3);
				\node[right,red,shift={(1.2,0)}] at (-1.4,3.5) {$T'$};
				
				\draw[fill=blue,blue] (-1.1,3) circle[radius=0.06];
				\draw[fill=blue,blue] (-2,4) circle[radius=0.06];
				\node[above] at (-0.9,3) {$D_i^+$};
				\draw (-1.65,3.4)--(-2.5,3.3);
				\node[left] at (-2.4,3.3) {$I\times G_1$};
				
				\draw[line width=2,blue] (-1.1,3)--(-2,4);
			\end{scope}
			
			\begin{scope}[shift={(1,-4)}]
				\draw[thick] (-4.3,3)--(-4.3,4) to [out=90,in=180] (-4,5)--(0.3,5);
				
				\draw[thick] (-4.3,3) to [out=-90,in=180] (-4,2)--(0.3,2);
				
				\draw[thick] (0,4) to [out=90,in=180] (0.3,5) to [out=0,in=90] (0.6,3.5) to [out=-90,in=0] (0.3,2) to [out=180,in=-90] (0,3)--(0,4);
				
				\draw[thick] (-4.3,3.5)--(-2.5,3.5) to[out=0,in=-115] (-2.3,3.8) to [out=75,in=180] (-2,5);

				\draw[thick] (0,3.5)--(-1.8,3.5) to[out=180,in=105] (-2.3,3) to [out=-75,in=180] (-2,2);
				
				\draw[thick,dotted] (-2,5) to[out=0,in=90] (-2,3.5) to[out=-90,in=0] (-2,2);
				
				\draw[thick,red] (-2.8,3)--(-2.8,4) to [out=85,in=-110] (-1.6,5);
				\draw[thick,red,dotted] (-1.6,5) to [out=-70,in=90] (-1.4,4) to [out=-90,in=70] (-1.4,3) to [out=-110,in=65] (-1.75,2);
				\draw[thick,red] (-1.75,2) to [out=115,in=-85] (-2.8,3);
				\node[red] at (-3,4.2) {$\widetilde{T}$};
				
				\node at (-3.5,2.5) {$Y$};

			\end{scope}
		\end{tikzpicture}
		\caption{(Adapted from \cite[Figure 9]{kronheimer2010knots}.) Cutting along $S$ and re-gluing back results in the same tangle closure $\widetilde{T}$.}\label{fig cutting}
	\end{figure}
	
 We now form a closure $Y'$ of $(M',\gamma')$. Cutting along $S$ results in $2n$ new sutures $D_i^\pm\,(i=1,2,\dots,n)$ corresponding to $C_i^\pm$. To form the closure, we use an auxiliary surface $G'=G\cup G_1$, where $G_1$ is a collection of $n$ annuli. We glue $G_1\times [-1,1]$ to $(M',\gamma')$ by identifying the sutures $(\partial G_1)\times \{0\}$ with $D_i^\pm$.	The preclosure has two boundary components \[\overline{R}'_\pm=\overline{R}^\dagger_\pm\cup S_\pm\cup G_1\times \{\pm 1\},\]where $\overline{R}^\dagger_\pm$ are obtained by cutting open $\overline{R}_\pm$ along the circles $C_i^\pm$, and $S_\pm$ are copies of $S$. We can then choose a diffeomorphism $h'\colon \overline{R}'_+\to\overline{R}'_-$ such that it coincides with $h$ on $\overline{R}^\dagger_-$ and equals to identify on $S_+\cup G_1\times\{+1\}$. The resulting closure $Y'$ contains a distinguished surface $\overline{R}'$. Let $T'$ be the result of cutting out $T$ along $S$. Then $T'$ is a representative of $\iota_*(\zeta)$. Let $\widetilde{T}'$ be the tangle closure of $T'$ in $Y'$, and then the action of $\iota_*(\zeta)$ can be identified with the action of $\mu(\widetilde{T}')$ on $I_*(Y'|\overline{R}')$.
 
 Kronheimer and Mrowka showed that $Y'$ is diffeomorphic to $Y$. More precisely, there is a diffeomorphism $\phi\colon Y'\to Y$ that restricts to the identity on $M$ and sends $\overline{R}'$ to the double-curve sum of $R$ and $\overline{S}$. Further, it sends $\widetilde{T}'$ to $\widetilde{T}$, as explained in Figure \ref{fig cutting}. In \cite[Proposition 7.11]{kronheimer2010knots}, it is showed that $I_*(Y'|\overline{R}')$ is a direct summand of $I_*(Y|R)$ according to the generalized-eigenspace decomposition \cite[Corollary 7.6]{kronheimer2010knots}. As the actions of $\mu([\widetilde{T}'])$ and $\mu([\widetilde{T}])$ commute with the surface actions, $I_*(Y'|\overline{R}')$ is also a direct summand of $I_*(Y|R)$ as $\CC[X]/X^2$-modules.
\end{proof}

\begin{proof}[Proof of Theorem \ref{thm splitting}]
	Assume first that $L$ is split. Then $S^3(L)$ is a connected sum of two balanced sutured manifolds $S^3(K_1)\# S^3(K_2)$, where $K_1$ and $K_2$ are two components of $L$. By Proposition \ref{prop connected sum}, we have an isomorphism of $\CC[X]/X^2$-modules\[\KHI(L)=\SHI(S^3(L))\cong \SHI(S^3(K_1))\otimes_\CC\SHI(S^3(K_2))\otimes_\CC\SHI(S^3(2)).\]The action of $X$ on the right hand side is given by $\operatorname{Id}\otimes\operatorname{Id}\otimes X_\zeta$. The last term $\SHI(S^3(2))$ is a free $\CC[X]/X^2$-module of rank $1$ as $\zeta$ has a non-zero intersection number with the  splitting sphere. Hence, $\KHI(L)$ is free as a $\CC[X]/X^2$-module by \cite[Lemma 2.9]{wang2021link}.
	
	If $L$ is not split, then the link complement $S^3(L)$ is taut. We can then form a sequence of nice surface decompositions
	\[(M,\gamma)=(M_1,\gamma_1)\overset{S_1}{\rsa}(M_2,\gamma_2){\rsa}\dots{\rsa}(M_n,\gamma_n)\]
	 to obtain a product balanced sutured manifold $(M_n,\gamma_n)$ at the end \cite{gabai1983foliations,juhasz2008floer}. Then $\SHI(M_n,\gamma_n)$ has rank $1$ as a $\CC$-module by \cite[Theorem 7.18]{kronheimer2010knots}. By Proposition \ref{prop sutured decomposition}, $\SHI(M_n,\gamma_n)$ is a direct summand of $\SHI(S^3(L))$ as $\CC[X]/X^2$-modules. As a rank $1$ $\CC$-module cannot be free as a $\CC[X]/X^2$-module, $\KHI(L)=\SHI(S^3(L))$ is not free either. 
\end{proof}

\bibliographystyle{hplain}
\bibliography{ref}

\end{document}